\documentclass[12pt,oneside,english,12pt]{amsart}
\usepackage[T1]{fontenc}
\usepackage[utf8]{inputenc}
\usepackage{amsthm}
\usepackage{amssymb}

\makeatletter
\numberwithin{equation}{section}
\numberwithin{figure}{section}
\theoremstyle{plain}
\newtheorem{thm}{Theorem}
 \theoremstyle{definition}
  \newtheorem{example}[thm]{Example}

\usepackage{amsthm}
\usepackage{diagrams}
\numberwithin{equation}{section} 
\numberwithin{figure}{section} 
\theoremstyle{plain}
 \theoremstyle{definition}
 \newtheorem{defn}[thm]{Definition}
 \newtheorem{prop}[thm]{Proposition}
 \newtheorem{rem}[thm]{Remark}
 \newtheorem{lemma}[thm]{Lemma}

\usepackage{amsfonts}

\usepackage{babel}

\usepackage{babel}

\makeatother

\usepackage{babel}

\begin{document}

\title{Paths on graphs and associated quantum groupoids}

\author{R. Trinchero}

\address{Centro Atómico Bariloche e Instituto Balseiro, Bariloche, Argentina.}

\email{trincher@cab.cnea.gov.ar}

\thanks{CONICET support is gratefully acknowledged. To appear in the proceedings
of \emph{Colloquium on Hopf Algebras, Quantum Groups and Tensor Categories},
August 31st to September 4th, 2009.}

\date{17 March 2010}
\begin{abstract}
Given any simple biorientable graph it is shown that there exists
a weak {*}-Hopf algebra constructed on the vector space of graded
endomorphisms of essential paths on the graph. This construction is
based on a direct sum decomposition of the space of paths into orthogonal
subspaces one of which is the space of essential paths. Two simple
examples are worked out with certain detail, the ADE graph $A_{3}$
and the affine graph $A_{[2]}$. For the first example the weak {*}-Hopf
algebra coincides with the so called double triangle algebra. No use
is made of Ocneanu's cell calculus. 
\end{abstract}
\maketitle

\section{Introduction}

One of the most interesting developments in mathematical physics of
the last decades has been the classification of SU(2)-type rational
conformal field theories by ADE graphs%
\footnote{Analog classifications exist for SU(3)-type\cite{Difzu} and SU(4)-type\cite{Ocbar}
rational conformal field theories, however the construction of the
corresponding weak Hopf algebras out of the analog of the ADE graphs
is not known.%
}\cite{capp1}. In relation to the present work a possible way to look
at this classification is the following%
\footnote{Another way closer to the historical path is given in \cite{Isaschi}%
}. The tensor category of representations of a weak {*}-Hopf algebra\cite{BS}
constructed out of the corresponding ADE graph $G$ is summarized
by another graph $Oc(G)$, called the Ocneanu graph of quantum symmetries\cite{Ocgoto}.
Knowledge of this last graph encodes information on the conformal
field theory when considered in various environments, the corresponding
generalized partition functions can be obtained from this graph\cite{capp1,CoqGil,CoqHuer}.
In addition the weak {*}-Hopf algebras mentioned above can be given
a physical interpretation as the algebras of quantum mechanical symmetries
of certain quantum statistical models, known as face models\cite{Trfdta}.

For the case of ADE graphs the weak {*}-Hopf algebra mentioned above
is known as the the double triangle algebra(DTA)\cite{Ocgoto,PetZub,CoqTr}.
The construction of this algebra out of the corresponding ADE graph
starts from something called quantum 6-j symbols\cite{coq6j} that
can be computed employing Ocneanu's cell calculus. These objects describe
the representation theory of the DTA\cite{Ocparag,Evkawa,Roche}.
No direct derivation of this weak Hopf algebra out of paths on the
corresponding graph is available in the literature. One of the aims
of this work is to fill this gap%
\footnote{The question of whether such a derivation exists or not was posed
by Oleg Ogievetsky in relation to joint work with the author.%
}.

The key ingredient in this derivation is a direct sum decomposition
of the space of paths of a given length into orthogonal subspaces,
one of which is the space of essential paths. The space of essential
paths can be defined in terms of a representation of the Temperley-Lieb-Jones
algebra in the space of paths over the graph. The other terms in the
above mentioned decomposition are obtained by means of the application
of Ocneanu creation operators to spaces of essential paths of a given
length. The product in the resulting weak Hopf algebra is defined
using a projection of the concatenation factor by factor of endomorphism
of paths. This projection sends graded endomorphism of paths into
graded endomorphism of essential paths.

The derivation mentioned above can be done for any simple bioriented
graph. This provides a generalization of the construction to simple
bioriented graphs that are not ADE. In that cases the resulting weak
{*}-Hopf algebra is infinite dimensional. For illustrative purposes
a pair of simple examples are considered in this work. One of which
is ADE and the other not.

Some interesting further research arise in relation to this work.
The representation theory of these weak {*}-Hopf algebras has not
been considered in this work. The detailed study of all the affine
graphs($\beta=2)$ weak {*}-Hopf algebras remains to be done. Also
the case of non-affine non-ADE graphs($\beta>2$) is missing. Furthermore
the relation of these weak {*}-Hopf algebras with conformal field
theory deserves to be considered. 

This paper is organized as follows. Sections 2,3 and 4 set up the
scenario and give the basic definitions. Section 5 presents the decomposition
and section 6 the projection mentioned above. Sections 7,8 and 9 deal
with the weak Hopf algebra structure.

\section{Paths\label{sec:Paths}}

Let $G$ denote a simple biorientable graph. Just to remind the reader
some basic definitions to be employed in what follows are included.%
\footnote{Further definitions and basic results on graph theory can be found
in any textbook on graph theory, a short account of these matters
related to this work are presented in appendix A of ref. \cite{Trfdta}%
}.\begin{defn} \emph{Adjacency matrix. }Let the graph\emph{ G }have
$N_{v}$ vertices, its\emph{ }adjacency matrix $M$ is the $N_{v}\times N_{v}$
matrix whose $v_{i}v_{j}$ entry is $1$ if the vertex $v_{i}$ is
connected to the vertex $v_{j}$ by an edge belonging to $G$, $0$
if it is not connected.\end{defn}

\begin{defn} \emph{Elementary path}, \emph{length}. A elementary
path is a succession of consecutive vertices in G. The number of these
vertices -1 is called the length of the path. \end{defn}

\begin{defn} \emph{\label{-Space-of paths}Space of paths $\mathcal{P}$}.
The inner product vector space of paths $\mathcal{P}$ is defined
by saying that elementary paths provide a orthonormal basis of this
space. \end{defn}

\begin{defn}\emph{Concatenation product in }$\mathcal{P}$. Given
two elementary path $\eta=(v_{0},v_{1},\cdots,v_{n})$ and $\eta'=(v_{0}',v_{1}',\cdots,v_{n}')$
their concatenation product $\eta\star\eta'$ is given by,\[
\eta\star\eta'=\delta_{v_{n}v_{0'}}(v_{0},v_{1},\cdots,v_{n},v_{1}',\cdots,v_{n}')\]
\end{defn}

\section{Creation and annihilation operators on $\mathcal{P}$}

Let $\eta=(v_{0},v_{1},\cdots,v_{n})$ denote a elementary path of
length $n$. \begin{defn} \emph{\label{def:Creation-and-annhilation}Creation
and annihilation operators $c_{i}^{\dagger}:\mathcal{P}_{n}\to\mathcal{P}_{n+2}$
and $c_{i}:\mathcal{P}_{n}\to\mathcal{P}_{n-2}$ \begin{eqnarray*}
c_{i}\,\eta & = & c_{i}\,(v_{0},v_{1},\cdots,v_{i},v_{i+1},v_{i+2},v_{i+3},\cdots,v_{n})\\
 & = & \delta_{v_{i}v_{i+2}}\sqrt{\frac{\mu_{v_{i+1}}}{\mu_{v_{i}}}}\,(v_{o},v_{1},\cdots,v_{i},v_{i+3},\cdots,v_{n})\;\; if\;0\leq i\leq n-2,\;\;0\; otherwise\end{eqnarray*}
} 

\begin{eqnarray}
c_{i}^{\dagger}\eta & = & c_{i}^{\dagger}(v_{0},v_{1},\cdots,v_{i},v_{i+1},\cdots,v_{n})\nonumber \\
 & = & \sum_{v\, n.n.v_{i}}\sqrt{\frac{\mu_{v}}{\mu_{v_{i}}}}\,(v_{0},v_{1},\cdots,v_{i},v,v_{i},v_{i+1},\cdots,v_{n})\;\; if\;0\leq i\leq n,\;\;0\; otherwise\label{eq:crea}\end{eqnarray}
where $\mathcal{P}_{n}$ is the inner product vector space of paths
of length $n$, $\mu_{v}$ denotes the components of the Perron-Frobenius
eigenvector%
\footnote{i.e., the eigenvector of the adjacency matrix $M$ with greatest eigenvalue
$\beta$ and with its smallest components taken to be $1$.%
} and $n.n.$ denotes nearest neighbours in $G$.\end{defn}

\begin{prop} \label{pro:For-,4}For $i\leq n$,\[
c_{i}c_{i}^{\dagger}=\beta1_{n}\]
where $\beta$ stands for the highest eigenvalue of the adjacency
matrix of G and $1_{n}$ denotes the identity in the space $\mathcal{P}_{n}$.
\end{prop} , \begin{prop} The following operators $e_{i}:\mathcal{P}_{n}\to\mathcal{P}_{n}\;\;,i=0,\cdots,n-2$
,\[
e_{i}=\frac{1}{\beta}c_{i}^{\dagger}c_{i}\]
 give a representation of the Temperley-Lieb-Jones algebra with $n-1$
generators. This algebra is defined by the following relations,\begin{eqnarray*}
e_{i}^{2}=e_{i}, & \;\; e_{i}^{\dagger}=e_{i},\;\; & e_{i}e_{j}=e_{j}e_{i}\;,|i-j|>1,\;\; e_{i}e_{i\pm1}e_{i}=\frac{1}{\beta^{2}}e_{i}\end{eqnarray*}

\end{prop}

\section{Essential paths}

\begin{defn}\label{sec:-Essential-Paths} \emph{Essential Paths Subspace.}
For each $n$ we denote by $\mathcal{E}_{n}$ the subspace of $\mathcal{P}_{n}$
defined by the relations,\[
\xi\in\mathcal{E}_{n}\Leftrightarrow c_{i}\xi=0,\; i=0,\cdots,n-2\]
 the essential paths subspace $\mathcal{E}$ is defined by,\[
\mathcal{E}=\bigoplus_{n}\mathcal{E}_{n}\]
 \end{defn}this definition implies that,

\begin{prop} For all the ADE graphs $\mathcal{E}$ is finite dimensional%
\footnote{See for example ref.\cite{Zubar} for a proof of this result.%
}. \end{prop}

In what follows we will denote by $\left\{ \xi_{a}\right\} $ an orthonormal
basis of $\mathcal{E}$ (with respect to the restriction to $\mathcal{E}$
of the scalar product in $\mathcal{P}$) , i.e. $(\xi_{a},\xi_{b})=\delta_{ab}$.
\begin{example}
\emph{\label{exa:Essentials-paths-for}Essentials paths for the graph
$A_{3}$ . }The graph $A_{3}$, its adjacency matrix, Perron-Frobenius
eigenvalue and eigenvector are given bellow,

\begin{figure}[h]
$${ \begin{diagram}[size=0.8em,abut] 
0 & \, & 1 & \, & 2 \\ 
\bullet&\rLine & \bullet &\rLine & \bullet \\ 
& \, & \, & \, &\, 
\end{diagram} }\;\;
\qquad , M=\left( \begin{array}{lll} 0&1 &0\\ 1&0 &1\\ 0&1 &0 \end{array} \right) ,\;\beta=\sqrt{2},\;\;\mu=\left( \begin{array}{l} 1 \\ \sqrt{2} \\ 1 \end{array}\right)$$
\caption{The graph $A_3$} 
\end{figure}\noindent there are ten essential paths in $A_{3}$, which are,

Length zero: $(0),(1),(2)$

Length one: $(01),(12),(10),(21)$

Length two: $(012),\gamma=\frac{1}{\sqrt{2}}[(121)-(101)],(210)$
\end{example}
Therefore the maximum length of essential paths over $A_{3}$ is $L=2$.
\begin{example}
\emph{Essentials paths for the graph $A_{[2]}$}. The graph $A_{[2]}$,
its adjacency matrix, Perron-Frobenius eigenvalue and eigenvector
are given bellow,

\begin{figure}[h]
$${ \begin{diagram}[size=0.8em,abut] 
&  & 1 &  &\\
&  & \bullet &  &  \\ 
&\ruLine & &\rdLine &  \\ 
\bullet & &\hLine   &  &\bullet \\
0 &  &  &  &2 
\end{diagram} }
\qquad , \qquad M\left( \begin{array}{lll} 0&1 &1\\ 1&0 &1\\ 1&1 &0 \end{array}\right)
,\;\beta=2,\;\;\mu=\left( \begin{array}{l} 1 \\ 1 \\ 1 \end{array}\right)
$$
\caption{The graph $A_{[2]}$} 
\end{figure}There are essential paths of any length in $A_{[2]}$. Any cyclic
sequence of consecutive vertices defines an essential path, as for
example $(120120contiguous120120\cdots)$.
\end{example}

\section{Decomposition of the space of paths}

\begin{defn}\emph{Maximum length of essential paths.} \label{L}$L$
is the maximum length of essential paths in a graph $G$ iff any path
of length greater than $L$ is necessarily non-essential.

\end{defn}

The following result will be used to write the above mentioned decomposition,

\begin{prop}\label{order}

\[
c_{i}^{\dagger}c_{j}^{\dagger}=c_{j+2}^{\dagger}c_{i}^{\dagger}\;\;\;\forall\, j\geq i\,(\Rightarrow c_{j}c_{i}=c_{i}c_{j+2}\;\;\;\forall\, j\geq i)\]

\end{prop}

\begin{proof}\begin{eqnarray*}
c_{i}^{\dagger}c_{j}^{\dagger}(v_{0},v_{1},\cdots,v_{i},\cdots,v_{j},\cdots,v_{n}) & = & \sum_{v}\sqrt{\frac{\mu_{v}}{\mu_{v_{j}}}}\; c_{i}^{\dagger}(v_{0},v_{1},\cdots,v_{i},\cdots,v_{j},v,v_{j},\cdots,v_{n})\\
 & = & \sum_{v,v'}\sqrt{\frac{\mu_{v}\mu_{v'}}{\mu_{v_{j}}\mu_{v_{i}}}}(v_{0},v_{1},\cdots,v_{i},v',v_{i},\cdots,v_{j},v,v_{j},\cdots,v_{n})\end{eqnarray*}
on the other hand,\begin{eqnarray*}
c_{j+2}^{\dagger}c_{i}^{\dagger}(v_{0},v_{1},\cdots,v_{i},\cdots,v_{j},\cdots,v_{n}) & = & \sum_{v'}\sqrt{\frac{\mu_{v'}}{\mu_{v_{i}}}}\; c_{j+2}^{\dagger}(v_{0},v_{1},\cdots,v_{i},v',v_{i},\cdots,v_{j},\cdots,v_{n})\\
 & = & \sum_{v,v'}\sqrt{\frac{\mu_{v}\mu_{v'}}{\mu_{v_{j}}\mu_{v_{i}}}}(v_{0},v_{1},\cdots,v_{i},v',v_{i},\cdots,v_{j},v,v_{j},\cdots,v_{n})\end{eqnarray*}

\end{proof}

\begin{prop}\label{sec:alge}The operators $c_{i}c_{j}^{\dagger}:\mathcal{P}_{n}\to\mathcal{P}_{n}$
satisfy,\begin{eqnarray}
c_{i}c_{j}^{\dagger} & = & c_{j-2}^{\dagger}c_{i}\;\quad if\;\; i<j-1\label{eq:2}\\
c_{i}c_{j}^{\dagger} & = & c_{j}^{\dagger}c_{i-2}\;\quad if\;\; i>j+1\label{eq:2'}\\
c_{i}c_{i\pm1}^{\dagger} & = & 1_{n}\;\;,i\pm1\leq n\label{eq:3}\\
c_{i}c_{i}^{\dagger} & = & \beta1_{n}\;\;,i\leq n\label{eq:4}\end{eqnarray}
 which imply, \begin{eqnarray}
c_{i}c_{j}^{\dagger} & = & (\beta\,\delta_{i,j}+\delta_{i,j+1}+\delta_{i,j-1})1_{n}+\theta(j-(i+2))c_{j-2}^{\dagger}c_{i}+\theta(-j+(i-2))c_{j}^{\dagger}c_{i-2}\label{cc}\nonumber\\
 & = & (\beta\,\delta_{i-j,0}+\delta_{i-j,1}+\delta_{i-j,-1})1_{n}+\theta(2-(i-j))c_{j-2}^{\dagger}c_{i}+\theta((i-j)-2)c_{j}^{\dagger}c_{i-2}\label{eq:cc1}\end{eqnarray}
 where $1_{n}$ is the identity operator in $\mathcal{P}_{n}$ and
the function $\theta$ is defined by,\[
\theta(i)=\begin{cases}
\begin{array}{l}
1\; if\; i\geq0\\
0\; otherwise\end{array}\end{cases}\]
 \end{prop} \begin{proof}It follows from definition \ref{def:Creation-and-annhilation}.\end{proof}
The second equality in (\ref{eq:cc1}) has been written to emphasize
the fact that the coefficients of the different terms depend only
on the difference $i-j$.

\begin{thm} The following decomposition holds%
\footnote{Each term in this decomposition can be characterized by the number
of non-essential back and forth subpaths. It has certain similarities
with what is called Fock's space in quantum field theory, however
they are quite different in some interesting respects. The role of
the vacuum is played here by essential paths, so the analogy would
be a theory with many non-equivalent vaccums, the number of which
could be infinite as for example in the case of $A_{[2]}$. Excitations
are created out of the vacuum by means of the creation operators $c_{i}^{\dagger}$.
The algebra of these creation and annihilation operators being given
by (\ref{eq:cc1}) which depends on the shape of the graph and which
differs significantly from the canonical one, which is given in terms
of commutators, appearing in the case of Fock's space.%
},\begin{eqnarray}
\mathcal{P}_{n} & = & \mathcal{E}_{n}\,\bigoplus_{i\leq n-2}c_{i}^{\dagger}(\mathcal{E}_{n-2})\,\bigoplus_{i_{1}<i_{2}\leq n-2}c_{i_{2}}^{\dagger}c_{i_{1}}^{\dagger}(\mathcal{E}_{n-4})\bigoplus\cdots\nonumber\\
 &  & \bigoplus_{i_{1}<i_{2}\cdots<i_{[n/2]}\leq n-2}c_{i_{[n/2]}}^{\dagger}c_{i_{[n/2]-1}}^{\dagger}\cdots c_{i_{1}}^{\dagger}(\mathcal{E}_{1|0})\label{eq:decomp-2}\end{eqnarray}
where $[]$ denotes the integer part and in the last summand one should
take $1$ for $n$ odd and $0$ for $n$ even. \end{thm}

\begin{proof}The following important lemma will be employed in this
proof,

\begin{lemma}For all $\eta\in\mathcal{P}_{n}$ such that $c_{i}(\eta)\neq0$
for some $i$ and $c_{j}(\eta)=0\;\forall\, j$ such that $i<j<n-2$
there exist coefficients $\alpha_{k}\;,k=i,\cdots,n$ such that,\begin{equation}
\eta=\sum_{k=i}^{n-2}\;\alpha_{k}\, c_{k}^{\dagger}(c_{i}(\eta))+\xi^{(i)}\label{eq:papa}\end{equation}
with $\xi^{(i)}$ satisfying $c_{j}(\xi^{(i)})=0\;\forall\, j$ such
that $i-1<j<n-2$. 

\end{lemma}\begin{proof}Consider the application of $c_{i}$ to
eq.(\ref{eq:papa}), \begin{eqnarray*}
c_{i}(\eta) & = & \sum_{k=i}^{n-2}\;\alpha_{k}\, c_{i}c_{k}^{\dagger}(c_{i}(\eta))+c_{i}(\xi^{(i)})\\
 & = & \sum_{k=i}^{n-2}\;\alpha_{k}\,\{(\beta\,\delta_{i,k}+\delta_{i,k-1})+\theta(k-(i+2))c_{k-2}^{\dagger}c_{i}](c_{i}(\eta))+c_{i}(\xi^{(i)})\\
 & = & (\beta\alpha_{i}+\alpha_{i+1})c_{i}(\eta)+c_{i}(\xi^{(i)})\end{eqnarray*}
where proposition \ref{sec:alge} was employed in the second equality
and proposition \ref{order} in the third. Therefore if we choose
$\alpha_{i}$ and $\alpha_{i+1}$ such that,\[
\beta\alpha_{i}+\alpha_{i+1}=1\]
then $c_{i}(\xi^{(i)})=0$. In general the application of $c_{i+l}\;,l=0,\cdots,n-2-i$
to eq. (\ref{eq:papa}) is considered,

\begin{eqnarray*}
c_{i+l}(\eta) & = & \sum_{k=i}^{n-2}\;\alpha_{k}\, c_{i+l}c_{k}^{\dagger}(c_{i}(\eta))+c_{i+l}(\xi^{(i)})\\
 & = & \sum_{k=i}^{n-2}\;\alpha_{k}\,\{(\beta\,\delta_{i+l,k}+\delta_{i+l,k-1}+\delta_{i+l,k+1})+\theta(2-(i+l-k))c_{k-2}^{\dagger}c_{i+l}\\
 &  & +\theta(l-2)\theta((i+l-k)-2))c_{k}^{\dagger}c_{i+l-2}](c_{i}(\eta))+c_{i+l}(\xi^{(i)})\\
 & = & (\beta\alpha_{i+l}+\alpha_{i+l+1}+\alpha_{i+l-1})c_{i}(\eta)+c_{i+l}(\xi^{(i)})\end{eqnarray*}
therefore if the coefficients $\alpha_{k}$ can be chosen such that,\[
\left(\begin{array}{c}
1\\
0\\
0\\
\vdots\\
0\end{array}\right)=\left(\begin{array}{ccccc}
\beta & 1 & 0 & \cdots & 0\\
1 & \beta & 1 & \cdots & 0\\
 & 1 & \beta & 1\\
 &  & \ddots & \ddots & \ddots\\
 &  &  & 1 & \beta\end{array}\right)\left(\begin{array}{c}
\alpha_{i}\\
\alpha_{i+1}\\
\alpha_{i+2}\\
\vdots\\
\alpha_{n-2}\end{array}\right)\]
then the result follows because $c_{i+l}(\eta)=0\;\;,l=1,\cdots,n-2-i$
by hypothesis. The determinant of this $(n-1-i)\times(n-1-i)$ matrix
can be calculated recursively%
\footnote{The same determinant appears in the calculation of the harmonic oscillator
transition probability using the path integral(see \cite{itzu}, p.
431). %
} leading to,\begin{equation}
D_{n-1-i}(\beta)=\det\left|\begin{array}{ccccc}
\beta & 1 & 0 & \cdots & 0\\
1 & \beta & 1 & \cdots & 0\\
 & 1 & \beta & 1\\
 &  & \ddots & \ddots & \ddots\\
 &  &  & 1 & \beta\end{array}\right|=\beta^{n-1-i}\,\frac{\lambda_{+}^{n-i}-\lambda_{-}^{n-i}}{\lambda_{+}-\lambda_{-}}\label{eq:dete}\end{equation}
where,\begin{equation}
\lambda_{\pm}=\frac{1\pm\sqrt{1-4\beta^{-2}}}{2}\label{eq:eigen}\end{equation}
for $\beta=2$ the two eigenvalues coincide, taking the limit $\beta\to2$
in (\ref{eq:dete})gives,\[
\lim_{\beta\to2}D_{n-1-i}(\beta)=(n-i)\]
which does not vanish for any $i$($i\leq n-2$). For $\beta\neq2$,
this determinant vanishes if,\begin{equation}
\lambda_{+}^{n-i}-\lambda_{-}^{n-i}=0\;\;\Rightarrow\left(\frac{\lambda_{+}}{\lambda_{-}}\right)^{n-i}=1\;\;(\beta\neq0)\label{eq:homog}\end{equation}
which has no solution for $\beta>2$. For the case $\beta<2$, the
ADE case, the eigenvalues are complex conjugate of each other, i.e.
$\lambda_{-}=\lambda_{+}^{*}$. From eq. (\ref{eq:eigen}) it is obtained,\[
\frac{\lambda_{+}}{\lambda_{-}}=e^{i\phi}\;\;,\phi\; such\; that\;\beta=2\cos\phi\]
but on the other hand for $\beta<2,\beta=2\cos\frac{\pi}{N}$ where
$N$ is the Coxeter number of $G$. However it is well known that
the maximum length of essential paths $L$ is related to the Coxeter
number by $L=N-2$, thus $\phi=\frac{\pi}{N}=\frac{\pi}{L+2}$, therefore
eq.(\ref{eq:homog})is the same as,\[
e^{i\frac{\pi(n-i)}{L+2}}=1\]
which can never be satisfied. This is so because under the assumptions
of this lemma the following inequality should hold $n-i-1\leq L$.
If it were not so then the path obtained by reversing $\eta$ and
including the first $n-i-1$ steps would be an essential path of length
greater than $L$ in the graph $G$ which is impossible by definition
of $L$.\end{proof}

Using this lemma the following algorithm can be employed to obtain
a unique decomposition of an arbitrary path $\eta\in\mathcal{P}_{n}$
as in the r.h.s. of (\ref{eq:decomp-2}). Decompose $\eta$ as in
(\ref{eq:papa}). Then decompose every $c_{i}(\eta)$ appearing in
the first term of the r.h.s. of (\ref{eq:papa}) using (\ref{eq:papa})
and do the same with $\xi^{(i)}$. At each step of this process the
resulting paths are annihilated by one more $c_{i}$ operator, since
the number of these operators that can act on an element of $\mathcal{P}_{n}$
is $n-2$ then this process necessarily converges to something belonging
to the r.h.s. of (\ref{eq:decomp-2}). The ordering of the indices
of the $c^{\dagger}$ operators in (\ref{eq:decomp-2}) follows using
proposition \ref{order}.\end{proof}

The following result is a simple consequence of the decomposition
(\ref{eq:decomp-2}).

\begin{prop}\label{ortdecomp}The subspaces of $\mathcal{P}_{n}$
given by,\[
P_{n}^{(l)}=\bigoplus_{i_{1}<i_{2}<\cdots\leq n-2}c_{i_{1}}^{\dagger}c_{i_{2}}^{\dagger}\cdots c_{i_{l}}^{\dagger}(\mathcal{E}_{n-2l})\,\;\;,P^{(0)}(n)=\mathcal{E}_{n},\;\; l=0,\cdots,[n/2]\]
are mutually orthogonal.\end{prop}
\begin{proof}
This proposition is proved if we show that,\[
M_{lm}=(c_{i_{1}}^{\dagger}c_{i_{2}}^{\dagger}\cdots c_{i_{l}}^{\dagger}(\xi^{(n-2l)}),c_{j_{1}}^{\dagger}c_{j_{2}}^{\dagger}\cdots c_{j_{m}}^{\dagger}(\xi^{(n-2m)}))\propto\delta_{lm}\;\;,\forall\xi^{(n-2l)}\in\mathcal{E}_{n-2l}\;,\xi^{(n-2m)}\in\mathcal{E}_{n-2m}\]
this in turn follows%
\footnote{See the proof of proposition \ref{me} for a similar argument.%
} from the relations in proposition \ref{sec:alge}.
\end{proof}
Thus there exist orthogonal projections on each of the subspaces $P^{(l)}\;\;,l=0,\cdots,[n/2]$
that we denote by $\Pi_{n}^{(l)}$ and satisfy $\Pi_{n}^{(l)}=\Pi_{n}^{(l)}{}^{2}=\Pi_{n}^{(l)}{}^{\dagger}$.
In particular $\Pi_{n}^{(0)}$ is a orthogonal projector over essential
paths of length $n$. 
\begin{example}
\emph{Decomposition of non-essential paths in $A_{3}$. }Using the
algorithm of the previous theorem the following decompositions of
non-essential paths of a given length coming from the concatenation
of essential paths are obtained,

Length two:

\begin{eqnarray*}
(01)\star(10)=(010) & = & \frac{1}{2^{1/4}}c_{0}^{\dagger}(0)\\
(21)\star(12)=(212) & = & \frac{1}{2^{1/4}}c_{0}^{\dagger}(2)\\
(10)\star(01)=(101) & = & \frac{1}{\sqrt{2}}(\frac{1}{2^{1/4}}c_{0}^{\dagger}(1)-\gamma)\\
(12)\star(21)=(121) & = & \frac{1}{\sqrt{2}}(\frac{1}{2^{1/4}}c_{0}^{\dagger}(1)+\gamma)\end{eqnarray*}
Length three%
\footnote{It is recalled that $\gamma=\frac{1}{\sqrt{2}}[(121)-(101)]$ as defined
in example \ref{exa:Essentials-paths-for}.%
}:\begin{eqnarray*}
(01)\star\gamma & = & (\frac{1}{2^{1/4}}c_{1}^{\dagger}-2^{1/4}c_{0}^{\dagger})(01)\\
(21)\star\gamma & = & -(\frac{1}{2^{1/4}}c_{1}^{\dagger}-2^{1/4}c_{0}^{\dagger})(21)\\
\gamma\star(10) & = & (\frac{1}{2^{1/4}}c_{0}^{\dagger}-2^{1/4}c_{1}^{\dagger})(10)\\
\gamma\star(12) & = & -(\frac{1}{2^{1/4}}c_{0}^{\dagger}-2^{1/4}c_{1}^{\dagger})(12)\\
(10)\star(012) & = & (2^{1/4}c_{0}^{\dagger}-\frac{1}{2^{1/4}}c_{1}^{\dagger})(12)\\
(012)\star(21) & = & (2^{1/4}c_{1}^{\dagger}-\frac{1}{2^{1/4}}c_{0}^{\dagger})(01)\\
(12)\star(210) & = & (2^{1/4}c_{0}^{\dagger}-\frac{1}{2^{1/4}}c_{1}^{\dagger})(10)\\
(210)\star(01) & = & (2^{1/4}c_{1}^{\dagger}-\frac{1}{2^{1/4}}c_{0}^{\dagger})(21)\end{eqnarray*}
Length four:

\begin{eqnarray}
(012)\star(210) & = & (c_{1}^{\dagger}-\frac{1}{\sqrt{2}}c_{2}^{\dagger})\, c_{0}^{\dagger}((0))\label{eq:de1}\\
(210)\star(012) & = & (c_{1}^{\dagger}-\frac{1}{\sqrt{2}}c_{2}^{\dagger})\, c_{0}^{\dagger}((2))\label{eq:de2}\\
\gamma\star\gamma & = & (c_{1}^{\dagger}-\frac{1}{\sqrt{2}}c_{2}^{\dagger})c_{0}^{\dagger}((1))\label{eq:de3}\end{eqnarray}

\end{example}
\emph{$\;$}
\begin{example}
\emph{\label{exa:Deca2}Decomposition of non-essential paths in $A_{[2]}$.
}Using the algorithm of the previous theorem the following decompositions
of non-essential paths of a given length coming from the concatenation
of essential paths are obtained,

Length two:

\begin{eqnarray*}
(01)\star(10)=\frac{1}{2}c_{0}^{\dagger}((0))+\frac{1}{2}\xi_{010}, & \;(02)\star(20)=\frac{1}{2}c_{0}^{\dagger}((0))-\frac{1}{2}\xi_{020}, & \;\;\xi_{010}=\xi_{020}=(010)-(020)\in\mathcal{E}\\
(12)\star(21)=\frac{1}{2}c_{0}^{\dagger}((1))+\frac{1}{2}\xi_{121}, & \;(10)\star(01)=\frac{1}{2}c_{0}^{\dagger}((1))-\frac{1}{2}\xi_{101}, & \;\;\xi_{121}=\xi_{101}=(121)-(101)\in\mathcal{E}\\
(20)\star(02)=\frac{1}{2}c_{0}^{\dagger}((2))+\frac{1}{2}\xi_{202}, & \;(21)\star(12)=\frac{1}{2}c_{0}^{\dagger}((2))-\frac{1}{2}\xi_{212}, & \;\;\xi_{202}=\xi_{212}=(202)-(212)\in\mathcal{E}\end{eqnarray*}
it is worth noting that the last two lines above can be obtained from
the first one by making cyclic permutations of the vertices $0,1$
and $2$(not for the indices of the $c^{\dagger}$ operators), i.e.
by applying the rotations contained in the symmetry group $C_{3v}$
of the graph $A_{[2]}$.

Length three:\begin{eqnarray*}
(10)\star(012) & =(1012)= & (\frac{2}{3}c_{0}^{\dagger}-\frac{1}{3}c_{1}^{\dagger})(12)+\xi_{1012}\\
(012)\star(21) & =(0121)= & (\frac{2}{3}c_{1}^{\dagger}-\frac{1}{3}c_{0}^{\dagger})(01)+\xi_{0121}\\
(10)\star\xi_{010} & =(1010)-(1020) & =(\frac{2}{3}c_{0}^{\dagger}-\frac{1}{3}c_{1}^{\dagger})(10)+\xi_{(10)\star\xi_{010}}\\
\xi_{010}\star(01) & =(0101)-(0201) & =(\frac{2}{3}c_{1}^{\dagger}-\frac{1}{3}c_{0}^{\dagger})(01)+\xi_{\xi_{010}\star(01)}\end{eqnarray*}
where,\begin{eqnarray*}
\xi_{1012} & = & \frac{1}{3}[(1012)-(1212)+(1202)]\;\in\mathcal{E}\\
\xi_{0121} & = & \frac{1}{3}[(0121)-(0101)+(0201)]\;\in\mathcal{E}\\
\xi_{(10)\star\xi_{010}} & = & \frac{2}{3}[(1010)-(1020)-(1210)]\;\in\mathcal{E}\\
\xi_{\xi_{010}\star(01)} & = & \frac{2}{3}[(0101)-(0201)-(0121)]\;\in\mathcal{E}\end{eqnarray*}
from these four decompositions and applying the elements of the symmetry
group $C_{3v}$ of the graph $A_{[2]}$ the other twenty decompositions
can be readily obtained.

Length four:

\begin{equation}
(01210)=(\frac{2}{3}\, c_{1}^{\dagger}-\frac{1}{3}c_{2}^{\dagger})\,[\frac{1}{2}c_{0}^{\dagger}((0))+\xi_{01210}^{(2)}]-(\frac{1}{2}c_{0}^{\dagger}-\frac{1}{3}c_{1}^{\dagger}+\frac{1}{6}c_{2}^{\dagger})\xi_{01210}^{(2)}+\xi_{01210}^{(0)}\label{eq:01210}\end{equation}

where $\xi^{(0)},\xi^{(2)}\in\mathcal{E}$ are given by,\begin{eqnarray*}
\xi_{01210}^{(0)} & = & \frac{1}{6}[(01210)+(02120)+(01020)-(02020)-(01010)+(02010)]\\
\xi_{01210}^{(2)} & = & \frac{1}{2}[(010)-(020)]\end{eqnarray*}
also,\begin{eqnarray*}
\xi_{010}\star(012) & = & [c_{1}^{\dagger}-\frac{1}{2}(c_{2}^{\dagger}+c_{0}^{\dagger})](012)+\xi_{\xi_{010}\star(012)}\\
(210)\star\xi_{010} & = & [c_{1}^{\dagger}-\frac{1}{2}(c_{2}^{\dagger}+c_{0}^{\dagger})](210)+\xi_{(210)\star\xi_{010}}\end{eqnarray*}
where $\xi_{\xi_{010}\star(012)},\xi_{(210)\star\xi_{010}}\in\mathcal{E}$
are given by,\begin{eqnarray*}
\xi_{\xi_{010}\star(012)} & = & \frac{1}{2}[(01012)-(02012)-(01212)+(01202)]\\
\xi_{(210)\star\xi_{010}} & = & \frac{1}{2}[(21010)-(21020)-(21210)+(20210)]\end{eqnarray*}
applying the elements of $C_{3v}$ to these decompositions the others
can be readily obtained. Finally,\[
\xi_{121}\star\xi_{121}=\frac{2}{3}c_{1}^{\dagger}c_{1}^{\dagger}(1)-\frac{1}{3}c_{2}^{\dagger}c_{1}^{\dagger}(1)+\xi_{\xi_{121}\star\xi_{121}}\]
where $\xi_{\xi_{121}\star\xi_{121}}\in\mathcal{E}$ is given by,\[
\xi_{\xi_{121}\star\xi_{121}}=\frac{2}{3}[(12121)+(10101)-(10121)-(12101)-(12021)-(10201)]\]

\end{example}

\section{The projection}

A posteriori motivation for the definition of the projection $P:End^{gr}(\mathcal{P})\to End^{gr}(\mathcal{E})$
appearing bellow is given by its properties with respect to the concatenation
product of paths (see propositions \ref{star}, \ref{asosp}, \ref{copp},
\ref{2copp}, \ref{coup}, \ref{antipp}). However it was proposed
based on its relation with the representation theory of these weak
{*}-Hopf algebras, representation theory that is not considered in
this paper. Before giving this definition a useful result is given,

\begin{prop}\label{me}\begin{equation}
(c_{i_{1}}c_{i_{2}}\cdots c_{i_{n}}c_{j_{n}}^{\dagger}c_{j_{n-1}}^{\dagger}\cdots c_{j_{1}}^{\dagger}\xi_{a},\xi_{b})=\delta_{ab}C(i_{1},\cdots;i_{n};j_{n},\cdots;j_{1})\label{eq:C}\end{equation}
where $\xi_{a},\xi_{b}$ are elements of an orthonormal basis of $\mathcal{E}$
such that $\#\xi_{a}=\#\xi_{b}\geq j_{1}$ and $\#\xi_{a}=\#\xi_{b}\geq i_{1}$.

\end{prop}

\begin{proof}The evaluation of the matrix element in the l.h.s. is
considered. By means of relation (\ref{eq:cc1}) the product of operators
$c_{i_{n}}c_{j_{n}}^{\dagger}$ is either replaced by a number $\beta$
or $1$(which we call a contraction), or they are interchanged with
a change in the index of one of them. In any case for the matrix element
to be non-vanishing, all the $i$ indices should be contracted with
$j$ indices. If this is not the case the matrix element vanishes
because necessarily a $c$ operator will be applied to $\xi_{a}$
or $\xi_{b}$ which gives zero because they are essential. 

\end{proof}

Let $End(\mathcal{P}_{n})$ ($End(\mathcal{E}_{n})$)denote the vector
space of endomorphism of length $n$ paths(essential paths)%
\footnote{In what follows the following equalities $End(\mathcal{P}_{n})=\mathcal{P}_{n}\otimes\mathcal{P}_{n}$
and $End(\mathcal{E}_{n})=\mathcal{E}_{n}\otimes\mathcal{E}_{n}$
will be employed. This is so because by means of the scalar product
appearing in definition \ref{-Space-of paths} and section \ref{sec:-Essential-Paths}
it is possible to identify $\mathcal{P}_{n}$ and $\mathcal{E}_{n}$
with its duals.%
}. In what follows the vector space of length preserving endomorphism
of paths(essential paths) will be considered. They are defined by,\[
End^{gr}(\mathcal{P})=\bigoplus_{n}End(\mathcal{P}_{n})\;\;,End^{gr}(\mathcal{E})=\bigoplus_{n}End(\mathcal{E}_{n})\]

\begin{defn}\label{A-projector-}A projector%
\footnote{In ref.\cite{coqgar} another projector $Q$ acting on the same vector
space is considered. Defining a product as in (\ref{eq:prod}) but
using $Q$ does not lead to a weak Hopf algebra structure.%
} $P:End^{gr}(\mathcal{P})\to End^{gr}(\mathcal{E})$ is defined by
its action on the terms appearing in the decomposition (\ref{eq:decomp-2}),\begin{eqnarray}
P(c_{j_{n}}^{\dagger}\cdots c_{j_{1}}^{\dagger}\xi_{a}\otimes c_{i_{n}}^{\dagger}\cdots c_{i_{1}}^{\dagger}\xi_{b}) & = & \sum_{\xi_{c}\in\mathcal{E}}\;(c_{i_{1}}c_{i_{2}}\cdots c_{i_{n}}c_{j_{n}}^{\dagger}c_{j_{n-1}}^{\dagger}\cdots c_{j_{1}}^{\dagger}\xi_{a},\xi_{c})\;\xi_{c}\otimes\xi_{b}\label{eq:proj}\nonumber\\
 & = & C(i_{1},\cdots,i_{n};j_{n},\cdots,j_{1})\;\xi_{a}\otimes\xi_{b}\end{eqnarray}
where $j_{1}<j_{2}<\cdots<j_{n}$ and $i_{1}<i_{2}<\cdots<i_{n}$
. \end{defn}It is clear that $P^{2}=P$ but $P^{\dagger}\neq P$
which implies that $P$ is not an orthogonal projection.

\begin{rem}It should be noted that the projection of an arbitrary
element $\eta\otimes\eta'$ of $End(\mathcal{P}_{n})$ is obtained
by applying definition \ref{A-projector-} to each term appearing
in the decomposition of $\eta\otimes\eta'$ as in eq. (\ref{eq:decomp-2}).
Thus in general this projection consists in a summation of elements
belonging to,\[
\bigoplus_{l=0}^{[n/2]}End(\mathcal{E}_{n-2l})\]
 Therefore it will not in general respect the grading.

\end{rem}

\begin{rem}Note that because of eq. (\ref{eq:C}) and the orthonormality
of the basis $\{\xi_{a}\}$, the following equality holds,\begin{eqnarray*}
P(c_{j_{n}}^{\dagger}c_{j_{n-1}}^{\dagger}\cdots c_{j_{1}}^{\dagger}\xi_{a}\otimes c_{i_{n}}^{\dagger}c_{i_{n-1}}^{\dagger}\cdots c_{i_{1}}^{\dagger}\xi_{b})=\\
=\sum_{\xi_{c}\in\mathcal{E}}\;\xi_{a}\otimes\xi_{c}\;(\xi_{c},c_{j_{1}}c_{j_{2}}\cdots c_{j_{n}}c_{i_{n}}^{\dagger}c_{i_{n-1}}^{\dagger}\cdots c_{i_{1}}^{\dagger}\xi_{b})\end{eqnarray*}

\end{rem}
\begin{example}
As an example the projections of the element $(01210)\otimes(21012)$
both for the graph $A_{3}$ and $A_{\{2\}}$ are calculated. For $A_{3}$,\begin{eqnarray*}
P((01210)\otimes(21012))_{A_{3}} & = & \sum_{v}((c_{1}^{\dagger}-\frac{1}{\sqrt{2}}c_{2}^{\dagger})\, c_{0}^{\dagger}((0)),(c_{1}^{\dagger}-\frac{1}{\sqrt{2}}c_{2}^{\dagger})\, c_{0}^{\dagger}((v)))\; v\otimes(2)\\
 & = & \sum_{v}(c_{0}(c_{1}-\frac{1}{\sqrt{2}}c_{2})\,(c_{1}^{\dagger}-\frac{1}{\sqrt{2}}c_{2}^{\dagger})\, c_{0}^{\dagger}((0)),v)\; v\otimes(2)\\
 & = & \sum_{v}\sqrt{2}(\sqrt{2}-\frac{1}{\sqrt{2}}-\frac{1}{\sqrt{2}}+\frac{\sqrt{2}}{2})\,\delta_{v,0}\; v\otimes(2)=0\otimes2\end{eqnarray*}
where the decompositions (\ref{eq:de1}) and (\ref{eq:de2}) were
employed for the first equality and proposition \ref{sec:alge} for
the last. For $A_{[2]}$, the decomposition of eq.(\ref{eq:01210})
and the following are employed,\begin{eqnarray*}
(21012) & = & (\frac{2}{3}\, c_{1}^{\dagger}-\frac{1}{3}c_{2}^{\dagger})\,[\frac{1}{2}c_{0}^{\dagger}((2))+\xi_{21012}^{(2)}]-(\frac{1}{2}c_{0}^{\dagger}-\frac{1}{3}c_{1}^{\dagger}+\frac{1}{6}c_{2}^{\dagger})\xi_{21012}^{(2)}+\xi_{21012}^{(0)}\end{eqnarray*}
where,\begin{eqnarray*}
\xi_{21012}^{(2)} & = & \frac{1}{2}[(212)-(202)]\\
\xi_{21012}^{(1)} & =\frac{1}{6}[ & (21012)+(20102)+(21202)-(20202)-(21212)+(20212)]\end{eqnarray*}
recalling that the projection kills terms with unequal number of $c^{\dagger}$
operators applied to essential paths in each factor of the tensor
product leads to, \begin{eqnarray*}
P((01210)\otimes(21012))_{A_{[2]}} & = & \sum_{\rho\in\mathcal{E}}(\xi_{01210}^{(1)},\rho)\,\rho\otimes\xi_{21012}^{(1)}\\
 &  & +\sum_{\rho\in\mathcal{E}}(([c_{1}^{\dagger}-\frac{1}{2}(c_{0}^{\dagger}+c_{2}^{\dagger})]\xi_{01210}^{(2)},[c_{1}^{\dagger}-\frac{1}{2}(c_{0}^{\dagger}+c_{2}^{\dagger})]\rho)\;\rho\otimes\xi_{21012}^{(2)}\\
 &  & +\sum_{v\in\mathcal{E}_{0}}((\frac{2}{3}\, c_{1}^{\dagger}-\frac{1}{3}c_{2}^{\dagger})\,\frac{1}{2}c_{0}^{\dagger}(0),(\frac{2}{3}\, c_{1}^{\dagger}-\frac{1}{3}c_{2}^{\dagger})\,\frac{1}{2}c_{0}^{\dagger}(v))\, v\otimes(2)\end{eqnarray*}
evaluating the scalar products gives,\[
P((01210)\otimes(21012))_{A_{[2]}}=\xi_{01210}^{(1)}\otimes\xi_{21012}^{(1)}+\xi_{01210}^{(2)}\otimes\xi_{21012}^{(2)}+\frac{1}{3}\,(0)\otimes(2)\]

\end{example}

\section{Star algebra}

In the vector space $End^{gr}(\mathcal{P})$ the following involution
is considered,

\begin{defn}\emph{Star.}

\begin{equation}
(\xi\otimes\xi')^{\star}=\xi^{\star}\otimes\xi'^{\star}\label{eq:star}\end{equation}
where $\xi^{\star}$ denotes the path obtained from $\xi$ by \textquotedbl{}time
inversion\textquotedbl{} for elementary paths and extending antilinearly
to all $\mathcal{P}$, i.e., by reversing the sense in which the succession
of contiguous vertices is followed for elementary paths , i.e.,\[
\xi=(v_{o},v_{1},\cdots,v_{n-1},v_{n})\Rightarrow\xi^{\star}=(v_{n},v_{n-1},\cdots,v_{1},v_{0})\]
\end{defn}

\noindent from this definition and the one of the scalar product in
section \ref{sec:Paths}, it is clear that,\begin{equation}
(\eta,\chi)=\overline{(\eta^{\star},\chi^{\star})}\label{eq:prodstar}\end{equation}
where the bar indicates the complex conjugate. The underlying vector
space of the algebra to be considered is given by the length graded
endomorphisms of essential paths%
\footnote{\textit{\emph{This choice of the underlying vector space structure
doe not mean that the product to be considered is the composition
of endomorphisms in $End^{gr}(\mathcal{E})$. It is emphasized that
this is }}\textit{not}\textit{\emph{ the product to be considered
but another product that we call $\cdot$ and that will be defined
bellow.}}%
}$End^{gr}(\mathcal{E})$. The product is defined by,

\begin{defn}\emph{Product.}

\begin{equation}
(\xi\otimes\xi')\cdot(\rho\otimes\rho')=P(\xi\star\rho\otimes\xi'\star\rho')\;\;\;;\xi,\rho,\xi',\rho'\in\mathcal{E}\label{eq:prod}\end{equation}
\end{defn}

\noindent This product does not make this algebra a graded one. This
is a filtered algebra respect to the length of paths. The product
of $\xi\otimes\xi'\in End(\mathcal{E}_{n_{1}})$ with $\rho\otimes\rho'\in End(\mathcal{E}_{n_{2}})$in
general belongs to,\[
\bigoplus_{l=0}^{[(n_{1}+n_{2})/2]}End(\mathcal{E}_{n_{1}+n_{2}-2l})\]

\noindent The identity is,\[
\mathbf{1}=\sum_{v,v'\in\mathcal{E}_{0}}\, v\otimes v'\]
the properties to be fulfilled by these definitions are fairly simple
to be proved except for the antihomomorphism property of the involution
(\ref{eq:star}) and the associativity of the product (\ref{eq:prod}).
The following result will be employed in the proof of the first of
these properties,

\begin{prop}\label{star}\begin{equation}
P((\eta\otimes\eta')^{\star})=P(\eta\otimes\eta')^{\star}\;\;\forall\eta,\eta'\in\mathcal{P}\label{eq:star-1}\end{equation}

\end{prop}

\begin{proof}

Typical contributions to the decomposition (\ref{eq:decomp-2}) for
$\eta$ and $\eta'$ are considered. These contributions should have
the same number of $c^{\dagger}$ operators applied to essential paths
in order to have a non-vanishing image when applying the projector.
Therefore the following expression is considered,\begin{eqnarray*}
P(c_{j_{n}}^{\dagger}c_{j_{n-1}}^{\dagger}\cdots c_{j_{1}}^{\dagger}\xi_{a}\otimes c_{i_{n}}^{\dagger}c_{i_{n-1}}^{\dagger}\cdots c_{i_{1}}^{\dagger}\xi_{b}) & = & \sum_{\xi_{c}\in\mathcal{E}}\;(c_{i_{1}}c_{i_{2}}\cdots c_{i_{n}}c_{j_{n}}^{\dagger}c_{j_{n-1}}^{\dagger}\cdots c_{j_{1}}^{\dagger}\xi_{a},\xi_{c})\;\xi_{c}\otimes\xi_{b}\\
 & = & \sum_{\xi_{c}\in\mathcal{E}}\;(c_{j_{n}}^{\dagger}c_{j_{n-1}}^{\dagger}\cdots c_{j_{1}}^{\dagger}\xi_{a},c_{i_{n}}^{\dagger}c_{i_{n-1}}^{\dagger}\cdots c_{i_{1}}^{\dagger}\xi_{c})\;\xi_{c}\otimes\xi_{b}\end{eqnarray*}
thus, \[
P(c_{j_{n}}^{\dagger}c_{j_{n-1}}^{\dagger}\cdots c_{j_{1}}^{\dagger}\xi_{a}\otimes c_{i_{n}}^{\dagger}c_{i_{n-1}}^{\dagger}\cdots c_{i_{1}}^{\dagger}\xi_{b})^{\star}=\sum_{\xi_{c}\in\mathcal{E}}\;\overline{(c_{j_{n}}^{\dagger}c_{j_{n-1}}^{\dagger}\cdots c_{j_{1}}^{\dagger}\xi_{a},c_{i_{n}}^{\dagger}c_{i_{n-1}}^{\dagger}\cdots c_{i_{1}}^{\dagger}\xi_{c})}\;\xi_{c}^{\star}\otimes\xi_{b}^{\star}\]
the time inversion of a path of the form $c_{i_{n}}^{\dagger}c_{i_{n-1}}^{\dagger}\cdots c_{i_{1}}^{\dagger}\xi_{b}$
leads to(the counting starts from the end of this path),\[
(c_{i_{n}}^{\dagger}\cdots c_{i_{2}}^{\dagger}c_{i_{1}}^{\dagger}\xi_{b})^{\star}=c_{l+2(n-1)-i_{n}}^{\dagger}c_{l+2(n-2)-i_{n-1}}^{\dagger}\cdots c_{l+2-i_{2}}^{\dagger}c_{l-i_{1}}^{\dagger}(\xi_{b}^{\star})\]
where $l=\#\xi_{b}$. Using this relation leads to,\begin{eqnarray*}
 & P((c_{j_{n}}^{\dagger}c_{j_{n-1}}^{\dagger}\cdots c_{j_{1}}^{\dagger}\xi_{a})^{\star}\otimes(c_{i_{n}}^{\dagger}c_{i_{n-1}}^{\dagger}\cdots c_{i_{1}}^{\dagger}\xi_{b})^{\star})=\\
= & \sum_{\xi_{d}^{\star}\in\mathcal{E}}((c_{j_{n}}^{\dagger}c_{j_{n-1}}^{\dagger}\cdots c_{j_{1}}^{\dagger}\xi_{a})^{\star},c_{l+2(n-1)-i_{n}}^{\dagger}\cdots c_{l+2-i_{2}}^{\dagger}c_{l-i_{1}}^{\dagger}\xi_{d}^{\star})\,\xi_{d}^{\star}\otimes\xi_{b}^{\star}\\
= & \sum_{\xi_{d}^{\star}\in\mathcal{E}}((c_{j_{n}}^{\dagger}c_{j_{n-1}}^{\dagger}\cdots c_{j_{1}}^{\dagger}\xi_{a})^{\star},(c_{i_{n}}^{\dagger}c_{i_{n-1}}^{\dagger}\cdots c_{i_{1}}^{\dagger}\xi_{d})^{\star})\,\xi_{d}^{\star}\otimes\xi_{b}^{\star}\\
= & \sum_{\xi_{d}\in\mathcal{E}}\overline{(c_{j_{n}}^{\dagger}c_{j_{n-1}}^{\dagger}\cdots c_{j_{1}}^{\dagger}\xi_{a},c_{i_{n}}^{\dagger}c_{i_{n-1}}^{\dagger}\cdots c_{i_{1}}^{\dagger}\xi_{d})}\,\xi_{d}^{\star}\otimes\xi_{b}^{\star}\end{eqnarray*}
where in the last equality eq. (\ref{eq:prodstar}) was employed and
the fact that when $\xi_{d}$ runs over all $\mathcal{E}$ then $\xi_{d}^{\star}$
also.

\end{proof}

Using (\ref{eq:star-1}) it follows that,

\begin{prop}\begin{equation}
((\xi\otimes\xi')\cdot(\rho\otimes\rho'))^{\star}=(\rho\otimes\rho')^{\star}\cdot(\xi\otimes\xi')^{\star}\;\;\forall\xi,\xi',\rho,\rho'\in\mathcal{E}\label{eq:star-2}\end{equation}

\end{prop}

\begin{proof}

\begin{eqnarray*}
(\rho\otimes\rho')^{\star}\cdot(\xi\otimes\xi')^{\star} & = & P((\rho\otimes\rho')^{\star}\star(\xi\otimes\xi')^{\star})=P(((\xi\otimes\xi')\star(\rho\otimes\rho'))^{\star})\\
 & = & P((\xi\otimes\xi')\star(\rho\otimes\rho'))^{\star}=((\xi\otimes\xi')\cdot(\rho\otimes\rho'))^{\star}\end{eqnarray*}

\end{proof}

In order to prove the associativity of the product (\ref{eq:prod})
the following preliminary result is considered,

\begin{prop}\label{asosp}\begin{eqnarray}
P((\xi\otimes\xi')\star P(\eta\otimes\eta')) & = & P((\xi\otimes\xi')\star(\eta\otimes\eta'))\label{eq:asos}\\
P(P(\eta\otimes\eta')\star(\xi\otimes\xi')) & = & P((\eta\otimes\eta')\star(\xi\otimes\xi'))\;\;\forall\xi,\xi'\in\mathcal{E},\;\;\eta,\eta'\in\mathcal{P}\label{eq:asos1}\end{eqnarray}

\end{prop}

\begin{proof}As in proposition \ref{star}, typical contributions
to the decomposition of $\eta\otimes\eta'$ are considered. Thus the
following expression is dealt with,\[
P((\xi\otimes\xi')\star P(c_{j_{n}}^{\dagger}c_{j_{n-1}}^{\dagger}\cdots c_{j_{1}}^{\dagger}\xi_{a}\otimes c_{i_{n}}^{\dagger}c_{i_{n-1}}^{\dagger}\cdots c_{i_{1}}^{\dagger}\xi_{b}))=C(i_{1},\cdots,i_{n};j_{n},\cdots,j_{1})\, P((\xi\otimes\xi')\star(\xi_{a}\otimes\xi_{b}))\]
where it was assumed that $P(c_{i_{n}}^{\dagger}c_{i_{n-1}}^{\dagger}\cdots c_{i_{1}}^{\dagger}\xi_{a}\otimes c_{j_{n}}^{\dagger}c_{j_{n-1}}^{\dagger}\cdots c_{j_{1}}^{\dagger}\xi_{b})$
does not vanish(if it vanishes it can be easily seen that the r.h.s.
of eq.(\ref{eq:asos}) also vanishes). Next the r.h.s. of eq.(\ref{eq:asos})
is considered,\begin{eqnarray*}
 & P((\xi\otimes\xi')\star(c_{j_{n}}^{\dagger}c_{j_{n-1}}^{\dagger}\cdots c_{j_{1}}^{\dagger}\xi_{a}\otimes c_{i_{n}}^{\dagger}c_{i_{n-1}}^{\dagger}\cdots c_{i_{1}}^{\dagger}\xi_{b}))=\\
= & P(\xi\star c_{j_{n}}^{\dagger}c_{j_{n-1}}^{\dagger}\cdots c_{j_{1}}^{\dagger}\xi_{a}\,\otimes\,\xi'\star c_{i_{n}}^{\dagger}c_{i_{n-1}}^{\dagger}\cdots c_{i_{1}}^{\dagger}\xi_{b})\\
= & P(c_{l+j_{n}}^{\dagger}\cdots c_{l+j_{1}}^{\dagger}(\xi\star\xi_{a})\otimes c_{l+i_{n}}^{\dagger}\cdots c_{l+i_{1}}^{\dagger}(\xi'\star\xi_{b}))\\
= & C(l+i_{1}\cdots,l+i_{n};l+j_{n},\cdots,l+j_{1})P((\xi\otimes\xi')\star(\xi_{a}\otimes\xi_{b}))\end{eqnarray*}
where $l$ denotes the length of the path $\xi$. From its definition
(\ref{eq:C}) it follows that,\[
C(i_{1},\cdots,i_{n};j_{n},\cdots,j_{1})=C(l+i_{1}\cdots,l+i_{n};l+j_{n},\cdots,l+j_{1})\]
which completes the proof of the first equality. Eq.(\ref{eq:asos1})
follows along identical lines.

\end{proof}

Using this result associativity follows,

\begin{prop}\[
((\xi_{1}\otimes\xi_{1}')\cdot(\xi_{2}\otimes\xi_{2}'))\cdot(\xi_{3}\otimes\xi_{3}')=(\xi_{1}\otimes\xi_{1}')\cdot((\xi_{2}\otimes\xi_{2}')\cdot(\xi_{3}\otimes\xi_{3}'))\;\;\forall\xi_{i},\xi_{i}'\in\mathcal{E\;},i=1,2,3\]

\end{prop}

\begin{proof}

\begin{eqnarray*}
((\xi_{1}\otimes\xi_{1}')\cdot(\xi_{2}\otimes\xi_{2}'))\cdot(\xi_{3}\otimes\xi_{3}') & = & P(P((\xi_{1}\otimes\xi_{1}')\star(\xi_{2}\otimes\xi_{2}'))\star(\xi_{3}\otimes\xi_{3}'))\\
 & = & P((\xi_{1}\otimes\xi_{1}')\star(\xi_{2}\otimes\xi_{2}')\star(\xi_{3}\otimes\xi_{3}'))\\
 & = & P((\xi_{1}\otimes\xi_{1}')\star P((\xi_{2}\otimes\xi_{2}')\star(\xi_{3}\otimes\xi_{3}')))\\
 & = & (\xi_{1}\otimes\xi_{1}')\cdot((\xi_{2}\otimes\xi_{2}')\cdot(\xi_{3}\otimes\xi_{3}'))\end{eqnarray*}

\end{proof}
\begin{example}
\emph{Product for the case of $A_{3}.$ }It can be explicitly verified
that in this case the product coincides with the one of the double
triangle algebra%
\footnote{In that reference product is calculated using the pairing with the
dual algebra, i.e. the 6j-symbols using the dual product (that is
the composition of endomorphisms) and coming back with the 6j-symbols.
The same construction for the case of $A_{[2]}$ is not known. The
extension of 6j-symbols(Ocneanu cells) for this last case is not obvious
since for example $A_{[2]}$ is not a bicolorable graph.%
} described in ref.\cite{CoqTr}. For illustrative purposes the calculation
of some of these products is given below,\begin{eqnarray*}
(21\otimes12)\cdot(12\otimes21) & = & P(212\otimes101)=\frac{1}{\sqrt{2}}P\left(\frac{1}{2^{1/4}}c_{0}^{\dagger}(2)\otimes(\frac{1}{2^{1/4}}c_{0}^{\dagger}(1)+\gamma)\right)=\frac{1}{\sqrt{2}}(2)\otimes(1)\\
(10\otimes12)\cdot(01\otimes21) & =\frac{1}{2} & P\left((\frac{1}{2^{1/4}}c_{0}^{\dagger}(1)-\gamma)\otimes(\frac{1}{2^{1/4}}c_{0}^{\dagger}(1)+\gamma)\right)=\frac{1}{2}(1\otimes1-\gamma\otimes\gamma)\\
(12\otimes12)\cdot(21\otimes21) & = & \frac{1}{2}P\left((\frac{1}{2^{1/4}}c_{0}^{\dagger}(1)+\gamma)\otimes(\frac{1}{2^{1/4}}c_{0}^{\dagger}(1)+\gamma)\right)=\frac{1}{2}(1\otimes1+\gamma\otimes\gamma)\\
(\gamma\otimes012)\cdot(10\otimes21) & = & P(\gamma\star10\otimes0121)=P\left((\frac{1}{2^{1/4}}c_{0}^{\dagger}-2^{1/4}c_{1}^{\dagger})(10)\otimes(2^{1/4}c_{1}^{\dagger}-\frac{1}{2^{1/4}}c_{0}^{\dagger})(01)\right)\\
 &  & =-(10)\otimes(01)\\
(\gamma\otimes\gamma)\cdot(\gamma\otimes\gamma) & = & P\left((c_{1}^{\dagger}-\frac{1}{\sqrt{2}}c_{2}^{\dagger})c_{0}^{\dagger}(1)\otimes(c_{1}^{\dagger}-\frac{1}{\sqrt{2}}c_{2}^{\dagger})c_{0}^{\dagger}(1)\right)=(1)\otimes(1)\end{eqnarray*}

\end{example}
$\;$
\begin{example}
\emph{Product for the case of $A_{[2]}.$ }With the notation of example
\ref{exa:Deca2} the following illustrative products can be computed,\begin{eqnarray*}
(21\otimes12)\cdot(12\otimes21) & = & P(212\otimes121)=P\left(\frac{1}{2}c_{0}^{\dagger}(2)-\frac{1}{2}\xi_{212}\otimes\frac{1}{2}c_{0}^{\dagger}(1)+\frac{1}{2}\xi_{121}\right)\\
 & = & \frac{1}{2}(2)\otimes(1)-\frac{1}{4}\xi_{212}\otimes\xi_{121}\\
(10\otimes12)\cdot(01\otimes21) & = & P(101\otimes121)=P\left((\frac{1}{2}c_{0}^{\dagger}(1)-\frac{1}{2}\xi_{101})\otimes(\frac{1}{2}c_{0}^{\dagger}(1)+\frac{1}{2}\xi_{121})\right)\\
 & = & \frac{1}{2}(1)\otimes(1)-\frac{1}{4}\xi_{121}\otimes\xi_{121})\\
(12\otimes12)\cdot(21\otimes21) & = & P(121\otimes121)=P\left((\frac{1}{2}c_{0}^{\dagger}(1)+\frac{1}{2}\xi_{121})\otimes(\frac{1}{2}c_{0}^{\dagger}(1)+\frac{1}{2}\xi_{121})\right)\\
 & = & \frac{1}{2}(1)\otimes(1)+\frac{1}{4}\xi_{121}\otimes\xi_{121})\\
(\xi_{121}\otimes012)\cdot(10\otimes21) & = & P(\xi_{121}\star10\otimes0121)\\
 & = & P\left((\frac{2}{3}c_{1}^{\dagger}-\frac{1}{3}c_{0}^{\dagger})(10)+\xi_{\xi_{121}\star10}\otimes(\frac{2}{3}c_{1}^{\dagger}-\frac{1}{3}c_{0}^{\dagger})(01)+\xi_{0121}\right)\\
 & = & \frac{2}{3}\,(10)\otimes(01)+\xi_{\xi_{121}\star10}\otimes\xi_{0121}\\
(\xi_{121}\otimes\xi_{121})\cdot(\xi_{121}\otimes\xi_{121}) & = & P\left((\frac{2}{3}c_{1}^{\dagger}c_{1}^{\dagger}-\frac{1}{3}c_{2}^{\dagger}c_{1}^{\dagger})(1)+\xi_{\xi_{121}\star\xi_{121}}\otimes(\frac{2}{3}c_{1}^{\dagger}c_{1}^{\dagger}-\frac{1}{3}c_{2}^{\dagger}c_{1}^{\dagger})(1)+\xi_{\xi_{121}\star\xi_{121}}\right)\\
 & = & \frac{4}{3}\,(1)\otimes(1)+\xi_{\xi_{121}\star\xi_{121}}\otimes\xi_{\xi_{121}\star\xi_{121}}\end{eqnarray*}
\emph{ }
\end{example}

\section{Weak bialgebra}

The definition of a weak $\star$-bialgebra is recalled,

\begin{defn}

A weak $\star$-bialgebra is a $\star$-algebra $A$ together with
two linear maps $\Delta:A\to A\otimes A$, the coproduct, and $\epsilon:\, A\mathbb{\to C}$,
the counit, satisfying the following axioms,\begin{eqnarray*}
\Delta(ab) & = & \Delta(a)\Delta(b)\\
\Delta(a^{\star}) & = & \Delta(a)^{\star}\\
(\Delta\otimes Id)\Delta & = & (Id\otimes\Delta)\Delta\end{eqnarray*}
and,\begin{eqnarray*}
\epsilon(ab) & = & \epsilon(a\mathbf{1_{1}})\epsilon(\mathbf{1_{2}}b)\\
(\epsilon\otimes Id)\Delta= & Id & =(Id\otimes\epsilon)\Delta\\
\epsilon(aa^{\star}) & \geq & 0\end{eqnarray*}
where in the first equation Sweedler convention is employed and also
in the following equation that defines $\mathbf{1_{1}}$ and $\mathbf{1_{2}}$,
\[
\Delta(\mathbf{1})=\mathbf{1_{1}}\otimes\mathbf{1_{2}}\]
with $\mathbf{1}$ being the identity in $A$ 

\end{defn}

The definition of coproduct and counit considered for the star algebra
of the previous section are,

\begin{defn}Coproduct%
\footnote{In the dual weak Hopf algebra to the one considered here, this coproduct
maps to the product. Eq. (\ref{eq:cop-1}) implies that this product
corresponds to the composition of endomorphisms in the dual weak Hopf
algebra.%
},\begin{equation}
\Delta(\xi\otimes\xi')=\sum_{\overset{\xi_{a}\in\mathcal{E}}{\#\xi_{a}=\#\xi}}\;\xi\otimes\xi_{a}\boxtimes\xi^{a}\otimes\xi'\label{eq:cop-1}\end{equation}

where the summation runs over a complete orthonormal basis for $\mathcal{E}.$

\end{defn}

\begin{defn}Counit,\[
\epsilon(\xi\otimes\xi')=(\xi,\xi')\]

\end{defn}

The axioms appearing in the definition of a weak bialgebra are fairly
simple to prove for the above definitions except for the morphism
property for the coproduct and the one involving the counit of a product.
For the first property the following preliminary results are useful,

\begin{prop}\label{copp}\begin{equation}
\Delta P=P^{\otimes2}\Delta_{\mathcal{P}}\label{eq:cop}\end{equation}
where $\Delta_{\mathcal{P}}(\chi\otimes\chi')=\sum_{\eta\in\mathcal{P}}\chi\otimes\eta\boxtimes\eta\otimes\chi'$
with summation over a complete orthonormal basis of $\mathcal{P}$.

\end{prop}

\begin{proof}

Eq. (\ref{eq:cop-1}) is applied to a generic element $\eta\otimes\eta'$of
$End(\mathcal{P})$, typical terms in the decomposition (\ref{eq:decomp-2})
with non-vanishing image by the projector are considered,\[
\Delta P(c_{j_{n}}^{\dagger}c_{j_{n-1}}^{\dagger}\cdots c_{j_{1}}^{\dagger}\xi_{a}\otimes c_{i_{n}}^{\dagger}c_{i_{n-1}}^{\dagger}\cdots c_{i_{1}}^{\dagger}\xi_{b})=\sum_{\xi_{c}\in\mathcal{E}.}\, C(i_{1},\cdots,i_{n};j_{n},\cdots,j_{1})\;\xi_{a}\otimes\xi_{c}\boxtimes\xi_{c}\otimes\xi_{b}\]
where it was assumed that $P(c_{j_{n}}^{\dagger}c_{j_{n-1}}^{\dagger}\cdots c_{j_{1}}^{\dagger}\xi\otimes c_{i_{n}}^{\dagger}c_{i_{n-1}}^{\dagger}\cdots c_{i_{1}}^{\dagger}\xi')$
does not vanish. On the other hand,\begin{eqnarray*}
 & P^{\otimes2}\Delta_{\mathcal{P}}(c_{j_{n}}^{\dagger}c_{j_{n-1}}^{\dagger}\cdots c_{j_{1}}^{\dagger}\xi_{a}\otimes c_{i_{n}}^{\dagger}c_{i_{n-1}}^{\dagger}\cdots c_{i_{1}}^{\dagger}\xi_{b})=\\
= & P^{\otimes2}(\sum_{\eta\in\mathcal{P}}c_{j_{n}}^{\dagger}c_{j_{n-1}}^{\dagger}\cdots c_{j_{1}}^{\dagger}\xi_{a}\otimes\eta\boxtimes\eta\otimes c_{i_{n}}^{\dagger}c_{i_{n-1}}^{\dagger}\cdots c_{i_{1}}^{\dagger}\xi_{b})\\
= & \sum_{\xi_{c},\xi_{d}\in\mathcal{E}}\,\sum_{\eta\in\mathcal{P}}(\xi_{c},c_{j_{1}}\cdots c_{j_{n}}\eta)(c_{i_{1}}\cdots c_{i_{n}}\eta,\xi_{d})\,\xi_{a}\otimes\xi_{c}\boxtimes\xi_{d}\otimes\xi_{b}\\
= & \sum_{\xi_{c},\xi_{d}\in\mathcal{E}}\,\sum_{\eta\in\mathcal{P}}(c_{j_{n}}^{\dagger}\cdots c_{j_{1}}^{\dagger}\xi_{c},\eta)(\eta,c_{i_{n}}^{\dagger}\cdots c_{i_{1}}^{\dagger}\xi_{d})\,\xi_{a}\otimes\xi_{c}\boxtimes\xi_{d}\otimes\xi_{b}\\
= & \sum_{\xi_{c},\xi_{d}\in\mathcal{E}}\,(c_{j_{n}}^{\dagger}\cdots c_{j_{1}}^{\dagger}\xi_{c},c_{i_{n}}^{\dagger}\cdots c_{i_{1}}^{\dagger}\xi_{d})\,\xi_{a}\otimes\xi_{c}\boxtimes\xi_{d}\otimes\xi_{b}\\
= & \sum_{\xi_{c},\xi_{d}\in\mathcal{E}}\,(\xi_{c},c_{j_{1}}\cdots c_{j_{n}}c_{i_{n}}^{\dagger}\cdots c_{i_{1}}^{\dagger}\xi_{d})\,\xi_{a}\otimes\xi_{c}\boxtimes\xi_{d}\otimes\xi_{b}\\
= & \sum_{\xi_{c}\in\mathcal{E}}\, C(i_{1},\cdots,i_{n};j_{n},\cdots,j_{1})\;\xi_{a}\otimes\xi_{c}\boxtimes\xi_{c}\otimes\xi_{b}\end{eqnarray*}

\end{proof}

\begin{prop}\label{2copp}

\[
P^{\otimes2}(\Delta_{\mathcal{P}}(\xi_{a}\otimes\xi_{b})\star\Delta_{\mathcal{P}}(\xi_{c}\otimes\xi_{d}))=P^{\otimes2}[P^{\otimes2}\Delta_{\mathcal{P}}(\xi_{a}\otimes\xi_{b})\star P^{\otimes2}\Delta_{\mathcal{P}}(\xi_{c}\otimes\xi_{d})]\]

\end{prop}

\begin{proof}

\begin{eqnarray*}
P^{\otimes2}(\Delta_{\mathcal{P}}(\xi_{a}\otimes\xi_{b})\star\Delta_{\mathcal{P}}(\xi_{c}\otimes\xi_{d})) & = & P^{\otimes2}(\sum_{\eta,\chi\in\mathcal{P}}(\xi_{a}\otimes\eta\boxtimes\eta\otimes\xi_{b})\star(\xi_{c}\otimes\chi\boxtimes\chi\otimes\xi_{d}))\\
 & = & P^{\otimes2}(\sum_{\eta,\chi\in\mathcal{P}}P(\xi_{a}\star\xi_{c}\otimes\eta\star\chi)\boxtimes P(\eta\star\chi\otimes\xi_{b}\star\xi_{d}))\end{eqnarray*}
on the other hand,\begin{eqnarray}
P^{\otimes2}[P^{\otimes2}\Delta_{\mathcal{P}}(\xi\otimes\xi')\star P^{\otimes2}\Delta_{\mathcal{P}}(\rho\otimes\rho')] & = & P^{\otimes2}[P^{\otimes2}\sum_{\eta,\chi\in\mathcal{P}}(\xi_{a}\otimes\eta\boxtimes\eta\otimes\xi_{b})\star P^{\otimes2}\Delta_{\mathcal{P}}(\xi_{c}\otimes\chi\boxtimes\chi\otimes\xi_{d})]\nonumber \\
 & = & P^{\otimes2}\sum_{\eta,\chi\in\mathcal{P}}P(P(\xi_{a}\otimes\eta)\star P(\xi_{c}\otimes\chi))\boxtimes P(P(\eta\otimes\xi_{b})\star P(\chi\otimes\xi_{d}))\label{eq:p2}\end{eqnarray}
next it is noted that,\begin{eqnarray*}
P(\xi_{a}\otimes\eta)\star P(\xi_{c}\otimes\chi) & = & \sum_{\omega,\sigma\in\mathcal{E}}(\omega,\eta)(\sigma,\chi)(\xi_{a}\otimes\omega)\star(\xi_{c}\otimes\sigma)\\
 & = & \sum_{\omega,\sigma\in\mathcal{E}}\,(\xi_{a}\star\xi_{c})\otimes(\omega\star\sigma)\,\delta_{\omega\eta}\delta_{\sigma\chi}\\
 & = & (\xi_{a}\star\xi_{c})\otimes(\eta\star\chi)\end{eqnarray*}
replacing in (\ref{eq:p2}) leads to,\[
P^{\otimes2}[P^{\otimes2}\Delta_{\mathcal{P}}(\xi_{a}\otimes\xi_{b})\star P^{\otimes2}\Delta_{\mathcal{P}}(\xi_{c}\otimes\xi_{d})]=P^{\otimes2}(\sum_{\eta,\chi\in\mathcal{P}}P(\xi_{a}\star\xi_{c}\otimes\eta\star\chi)\boxtimes P(\eta\star\chi\otimes\xi_{b}\star\xi_{d}))\]

\end{proof}

Using the above results leads to,

\begin{prop}

\[
\Delta((\xi\otimes\xi')\cdot(\rho\otimes\rho'))=\Delta(\xi\otimes\xi')\cdot\Delta(\rho\otimes\rho')\]

\end{prop}

\begin{proof}

\begin{eqnarray*}
\Delta((\xi\otimes\xi')\cdot(\rho\otimes\rho')) & = & \Delta(P((\xi\otimes\xi')\star(\rho\otimes\rho')))\\
 & = & P^{\otimes2}\Delta_{\mathcal{P}}((\xi\otimes\xi')\star(\rho\otimes\rho'))\\
 & = & P^{\otimes2}(\Delta_{\mathcal{P}}(\xi\otimes\xi')\star\Delta_{\mathcal{P}}(\rho\otimes\rho'))\\
 & = & P^{\otimes2}[P^{\otimes2}\Delta_{\mathcal{P}}(\xi\otimes\xi')\star P^{\otimes2}\Delta_{\mathcal{P}}(\rho\otimes\rho')]\\
 & = & P^{\otimes2}[\Delta(P(\xi\otimes\xi'))\star\Delta(P(\rho\otimes\rho'))]\\
 & = & \Delta(\xi\otimes\xi')\cdot\Delta(\rho\otimes\rho')\end{eqnarray*}

\end{proof}

Regarding the counit of a product the following result will be employed,

\begin{prop}\label{coup}

\[
\epsilon(P(\eta\otimes\eta'))=\epsilon(\eta\otimes\eta')\;\;\forall\,\eta,\eta'\in\mathcal{P}\]

\end{prop}

\begin{proof}

A generic element $\eta\otimes\eta'$of $End(\mathcal{P})$ is considered,
typical terms in the decomposition (\ref{eq:decomp-2}) with non-vanishing
image by the projector are considered,

\begin{eqnarray*}
 &  & \epsilon(P(c_{j_{n}}^{\dagger}c_{j_{n-1}}^{\dagger}\cdots c_{j_{1}}^{\dagger}\xi_{a}\otimes c_{i_{n}}^{\dagger}c_{i_{n-1}}^{\dagger}\cdots c_{i_{1}}^{\dagger}\xi_{b}))=\\
 &  & =\epsilon(\sum_{\xi_{c}\in\mathcal{E}}\;(c_{i_{1}}c_{i_{2}}\cdots c_{i_{n}}c_{j_{n}}^{\dagger}c_{j_{n-1}}^{\dagger}\cdots c_{j_{1}}^{\dagger}\xi_{a},\xi_{c})\;\xi_{c}\otimes\xi_{b})\\
 &  & =(c_{i_{1}}c_{i_{2}}\cdots c_{i_{n}}c_{j_{n}}^{\dagger}c_{j_{n-1}}^{\dagger}\cdots c_{j_{1}}^{\dagger}\xi_{a},\xi_{b})\\
 &  & =\epsilon(c_{j_{n}}^{\dagger}c_{j_{n-1}}^{\dagger}\cdots c_{j_{1}}^{\dagger}\xi_{a}\otimes c_{i_{n}}^{\dagger}c_{i_{n-1}}^{\dagger}\cdots c_{i_{1}}^{\dagger}\xi_{b})\end{eqnarray*}

\end{proof}

thus,

\begin{prop}

\[
\epsilon((\xi_{a}\otimes\xi_{b})\cdot(\xi_{c}\otimes\xi_{d}))=\epsilon((\xi_{a}\otimes\xi_{b})\cdot\mathbf{1}_{1})\epsilon(\mathbf{1}_{2}\cdot(\xi_{c}\otimes\xi_{d}))\]

\end{prop}

\begin{proof}

\[
\epsilon((\xi_{a}\otimes\xi_{b})\cdot(\xi_{c}\otimes\xi_{d}))=\epsilon(P(\xi_{a}\star\xi_{c}\otimes\xi_{b}\star\xi_{d}))=\epsilon(\xi_{a}\star\xi_{c}\otimes\xi_{b}\star\xi_{d})=(\xi_{a}\star\xi_{c},\xi_{b}\star\xi_{d})\]
on the other hand,\begin{eqnarray*}
\epsilon((\xi_{a}\otimes\xi_{b})\cdot\mathbf{1}_{1})\epsilon(\mathbf{1}_{2}\cdot(\xi_{c}\otimes\xi_{d})) & = & \sum_{v,u,v'}\epsilon((\xi_{a}\otimes\xi_{b})\cdot(v\otimes u))\epsilon((u\otimes v')\cdot(\xi_{c}\otimes\xi_{d}))\\
 & = & \sum_{v,u,v'}\;\delta_{r(\xi_{a})v}\delta_{r(\xi_{b})u}\delta_{us(\xi_{c})}\delta_{v's(\xi_{d})}(\xi_{a},\xi_{b})(\xi_{c},\xi_{d})\\
 & = & (\xi_{a}\star\xi_{c},\xi_{b}\star\xi_{d})=\epsilon((\xi_{a}\otimes\xi_{b})\cdot(\xi_{c}\otimes\xi_{d}))\end{eqnarray*}

\end{proof}

\section{The antipode}

In general the axioms to be satisfied by the antipode are,\begin{eqnarray}
S(ab) & = & S(b)S(a)\nonumber \\
S((S(a^{\star})^{\star}) & = & a\nonumber \\
\Delta(S(a)) & = & S\otimes S\,(\Delta^{op}(a))\nonumber \\
S(a_{1})\cdot a_{2}\otimes a_{3} & = & \mathbf{1}_{1}\otimes a\mathbf{1}_{2}\label{eq:axant}\end{eqnarray}
where in the last equation Sweedler convention has been employed. 

The following ansatz for the antipode is considered,\begin{equation}
S(\xi\otimes\omega)=F(\xi,\omega)\;\omega^{\star}\otimes\xi^{\star}\label{eq:antip}\end{equation}
where $F(\xi,\omega)$ is a numerical factor to be determined. It
is fairly simple to show that the first three axioms in (\ref{eq:axant})
are satisfied by this definition. The proof of the last axiom is more
involved. The following preliminary results are considered,

\begin{prop}\label{forthback-1}The following holds,

\begin{equation}
c_{n-1}^{\dagger}c_{n-2}^{\dagger}\cdots c_{0}^{\dagger}(v_{0})=\sum_{\eta\in\mathcal{P}_{n}/s(\eta)=v_{0}}\,\sqrt{\frac{\mu_{r(\eta)}}{\mu_{s(\eta)}}}\;\eta\star\eta^{\star}\label{eq:backforth}\end{equation}
where the summation is over the orthonormal basis of elementary paths
with starting vertex $v_{0}$ and,\begin{equation}
(\Pi_{n}^{(0)}\star\Pi_{n}^{(0)})c_{n-1}^{\dagger}c_{n-2}^{\dagger}\cdots c_{0}^{\dagger}(v_{0})=\sum_{\xi\in\mathcal{E}_{n}/s(\eta)=v_{0}}\,\sqrt{\frac{\mu_{r(\xi)}}{\mu_{s(\xi)}}}\;\xi\star\xi^{\star}\label{eq:projbf}\end{equation}
where $\Pi_{n}^{(0)}$ is the orthogonal projector over essential
paths of length $n$ mentioned after proposition \ref{ortdecomp},
the notation $(\Pi_{n}^{(0)}\star\Pi_{n}^{(0)})$ indicates that when
applied to a path of length $2n$ this operator projects over paths
that are essential in its first $n$ steps and also essential in its
last $n$ steps.

\end{prop}

\begin{proof}It follows from definition (\ref{eq:crea}).

\end{proof}

\begin{prop}\label{cbf}Let $\xi,\rho\in\mathcal{E}_{n}$, then,
\[
c_{i_{1}}\cdots c_{i_{n}}\xi^{\star}\star\rho=\delta_{i_{n}n-1}\cdots\delta_{i_{1}0}\;\delta_{\rho\xi}\sqrt{\frac{\mu_{v_{n}^{\xi^{\star}}}}{\mu_{v_{0}^{\xi^{\star}}}}}s(\xi^{\star})\]

\end{prop}

\begin{proof}Since $\xi^{\star},\rho\in\mathcal{E}$ then the only
$c$ operator that could give a non-zero result when applied the path
$\xi^{\star}\star\rho$ is $c_{n-1}$(thus $i_{m}=n-1$), indeed it
gives a non-zero result only if given a certain elementary path $\xi_{I}^{\star}$
appearing in the expression of $\xi^{\star}$there is a corresponding
elementary path $\rho_{I}$ appearing in the expression of $\rho$
such that the first step in $\xi_{I}$ (i.e. the inverse of the last
step of $\xi_{I}^{\star}$) coincides with the first step of $\rho_{I}$.
More precisely if $\xi_{I}^{\star}=(v_{0}^{\xi^{\star}},v_{1}^{\xi^{\star}},\cdots,v_{n}^{\xi^{\star}})$
and $\rho_{I}=(v_{0}^{\rho},v_{1}^{\rho},\cdots,v_{n}^{\rho})$ then,\begin{eqnarray*}
c_{n-1}(\xi_{I}^{\star}\star\rho_{I}) & = & \delta_{v_{n}^{\xi^{\star}}v_{0}^{\rho}}\,\delta_{v_{n-1}^{\xi^{\star}}v_{1}^{\rho}}\,\sqrt{\frac{\mu_{v_{n}^{\xi^{\star}}}}{\mu_{v_{n-1}^{\xi^{\star}}}}}(v_{0}^{\xi^{\star}},v_{1}^{\xi^{\star}},\cdots,v_{n-1}^{\xi^{\star}},v_{1}^{\rho},\cdots,v_{n}^{\rho})\\
 & = & \delta_{v_{n}^{\xi^{\star}}v_{0}^{\rho}}\,\delta_{v_{n-1}^{\xi^{\star}}v_{1}^{\rho}}\,\sqrt{\frac{\mu_{v_{n}^{\xi^{\star}}}}{\mu_{v_{n-1}^{\xi^{\star}}}}}\,\xi_{I,n-1}^{\star}\star\rho_{I,n-1}\end{eqnarray*}
the first delta function appears because the concatenation $\xi^{\star}\star\rho$
should not vanish, the second from the definition of the $c$ operator
and the last equality is just a definition of the path $\xi_{I,n-1}^{\star}\star\rho_{I,n-1}$.
Next consider the application of a $c$-operator to $\xi_{I,n-1}^{\star}\star\rho_{I,n-1}$,
in a similar fashion, only $c_{n-2}$ (thus $i_{m-1}=n-2$) gives
a non-zero result, which is,\begin{eqnarray*}
c_{n-2}(\xi_{I,n-1}^{\star}\star\rho_{I,n-1}) & = & \delta_{v_{n-2}^{\xi^{\star}}v_{2}^{\rho}}\sqrt{\frac{\mu_{v_{n-1}^{\xi^{\star}}}}{\mu_{v_{n-2}^{\xi^{\star}}}}}(v_{0}^{\xi^{\star}},v_{1}^{\xi^{\star}},\cdots,v_{n-2}^{\xi^{\star}},v_{2}^{\rho},\cdots,v_{n}^{\rho})\\
 & = & \delta_{v_{n-2}^{\xi^{\star}}v_{2}^{\rho}}\sqrt{\frac{\mu_{v_{n-1}^{\xi^{\star}}}}{\mu_{v_{n-2}^{\xi^{\star}}}}}\,\xi_{I,n-2}^{\star}\star\rho_{I,n-2}\end{eqnarray*}
proceeding in this way and collecting the contribution of each elementary
term, finally leads to,\[
c_{i_{1}}\cdots c_{i_{m}}\xi^{\star}\star\rho=\delta_{i_{m}n-1}\cdots\delta_{i_{1}0}\;\delta_{\rho\xi}\sqrt{\frac{\mu_{v_{n}^{\xi^{\star}}}}{\mu_{v_{0}^{\xi^{\star}}}}}s(\xi^{\star})\]

\end{proof}

Using the above result leads to,

\begin{prop}\label{antipp}Definition (\ref{eq:antip}) satisfies
(\ref{eq:axant}) with,\[
F(\xi,\omega)=\sqrt{\frac{\mu_{s(\omega)}\mu_{r(\xi)}}{\mu_{r(\omega)}\mu_{s(\xi)}}}\]

\end{prop}

\begin{proof}Replacing the ansatz (\ref{eq:antip}) in the last axiom
in (\ref{eq:axant}), leads to,\begin{equation}
\sum_{\xi_{c},\xi_{d}\in\mathcal{E}}\; F(\xi,\xi_{c})(\xi_{c}^{\star}\otimes\xi^{\star})\cdot(\xi_{c}\otimes\xi_{d})\boxtimes\xi_{d}\otimes\omega=\sum_{v,u,v'\in\mathcal{E}_{0}}\; v\otimes u\boxtimes(\xi\otimes\omega)\cdot(u\otimes v')\label{eq:axant1}\end{equation}
employing the definition of the product and the fact that $(\xi\otimes\omega)\cdot(u\otimes v')=\delta_{r(\xi)u}\delta_{r(\omega)v'}(\xi\otimes\omega)$
shows that (\ref{eq:axant1}) is equivalent to,\begin{equation}
\sum_{\xi_{c}\in\mathcal{E}}\; F(\xi,\xi_{c})P(\xi_{c}^{\star}\star\xi_{c}\otimes\xi^{\star}\star\xi_{d})=\delta_{\xi_{d}\xi}\sum_{v\in\mathcal{E}_{0}}\; v\otimes r(\xi)\label{eq:ant}\end{equation}
Choosing the factor $F(\xi,\xi_{c})$ to be of the form,\[
F(\xi,\xi_{c})=\alpha(\xi)\sqrt{\frac{\mu_{v_{0}^{\xi_{c}}}}{\mu_{v_{n}^{\xi_{c}}}}}\]
the l.h.s. of this last equation is given by,\begin{eqnarray*}
\sum_{\xi_{c}\in\mathcal{E}}\; F(\xi,\xi_{c})P(\xi_{c}^{\star}\star\xi_{c}\otimes\xi^{\star}\star\xi_{d}) & = & \alpha(\xi)\sum_{v(=r(\xi_{c}))}P((\Pi_{n}^{(0)}\star\Pi_{n}^{(0)})c_{n-1}^{\dagger}c_{n-2}^{\dagger}\cdots c_{0}^{\dagger}(v)\otimes\xi^{\star}\star\xi_{d})\\
 & = & \alpha(\xi)\sum_{v\in\mathcal{E}_{0},\rho\in\mathcal{E}}v\otimes\rho((\Pi_{n}^{(0)}\star\Pi_{n}^{(0)})c_{n-1}^{\dagger}c_{n-2}^{\dagger}\cdots c_{0}^{\dagger}\rho,\xi^{\star}\star\xi_{d})\\
 & = & \alpha(\xi)\sum_{v\in\mathcal{E}_{0},\rho\in\mathcal{E}}v\otimes\rho(\rho,c_{0}\cdots c_{n-2}c_{n-1}(\Pi_{n}^{(0)}\star\Pi_{n}^{(0)})\xi^{\star}\star\xi_{d})\\
 & = & \alpha(\xi)\sum_{v\in\mathcal{E}_{0},\rho\in\mathcal{E}}v\otimes\rho(\rho,r(\xi))\delta_{\xi\xi_{d}}\sqrt{\frac{\mu_{s(\xi)}}{\mu_{r(\xi)}}}\\
 & = & \alpha(\xi)\sqrt{\frac{\mu_{s(\xi)}}{\mu_{r(\xi)}}}\delta_{\xi\xi_{d}}\;\sum_{v\in\mathcal{E}_{0}}\, v\otimes r(\xi)\end{eqnarray*}
where in the first equality we have employed (\ref{eq:projbf}) of
proposition \ref{forthback-1}, the second equality involves the definition
of the projector $P$, the hermiticity of the projector $(\Pi_{n}^{(0)}\star\Pi_{n}^{(0)})$
was employed in writing the third equality, the fact that $\xi^{\star}$
and $\xi_{d}$ are already essential and proposition \ref{cbf} were
employed in the fourth equality. Thus choosing,\[
\alpha(\xi)=\sqrt{\frac{\mu_{r(\xi)}}{\mu_{s(\xi)}}}\Rightarrow F(\xi,\xi_{c})=\sqrt{\frac{\mu_{r(\xi)}\mu_{s(\xi_{c})}}{\mu_{s(\xi)}\mu_{r(\xi_{c})}}}\]
 leads to the result.

\end{proof}

\end{document}